\documentclass{amsart}
\usepackage{graphicx}
\usepackage{amsmath}
\usepackage{amsfonts}
\usepackage{amssymb}

\newtheorem{theorem}{Theorem}

\newtheorem{condition}[theorem]{Condition}

\newtheorem{definition}[theorem]{Definition}
\newtheorem{example}[theorem]{Example}

\newtheorem{lemma}[theorem]{Lemma}

\newtheorem{proposition}[theorem]{Proposition}
\newtheorem{remark}[theorem]{Remark}

\begin{document}

\title[Symmetrization inequalities for convex profile]{Symmetrization inequalities for probability metric spaces with convex isoperimetric profile}
\author{Joaquim Mart\'{i}n$^{\ast}$}
\address{Department of Mathematics\\
Universitat Aut\`{o}noma de Barcelona} \email{jmartin@mat.uab.cat}
\author{Walter A.  Ortiz**}
\address{Department of Mathematics\\
Universitat Aut\`{o}noma de Barcelona} \email{waortiz@mat.uab.cat}
\thanks{$^{\ast}$Partially supported by Grants MTM2016-77635-P, MTM2016-75196-P (MINECO) and 2017SGR358}
\thanks{**Partially supported by Grant MINECO MTM2016-77635-P}
\thanks{This paper is in final form and no version of it will be submitted for
publication elsewhere.} 

\begin{abstract}
We obtain symmetrization inequalities on probability metric spaces
with convex isoperimetric profile which incorporate in their
formulation the isoperimetric estimator and that can be applied to
provide a unified treatment of sharp Sobolev-Poincar\'{e} and Nash
inequalities.
\end{abstract}
\maketitle
\section{Introduction}

Let $\left(  \Omega,d,\mu\right)  $ be a connected metric space
equipped with a separable Borel probability measure $\mu$. The
perimeter or Minkowski content of a Borel set $A\subset\Omega,$ is
defined by
\[
\mu^{+}(A)=\liminf_{h\rightarrow0}\frac{\mu\left(  A_{h}\right)  -\mu\left(
A\right)  }{h},
\]
where $A_{h}=\left\{  x\in\Omega:d(x,A)<h\right\}$ is the open
$h-$neighborhood of $A.$ The \textbf{isoperimetric profile}
$I_{\mu}$ is defined as the pointwise maximal function
$I_{\mu}:[0,1]\rightarrow\left[  0,\infty\right)  $ such that
\[
\mu^{+}(A)\geq I_{\mu}\left(  \mu(A)\right)  ,
\]
holds for all Borel sets $A$. An isoperimetric inequality measures
the relation between the boundary measure and the measure of a set,
by providing a lower bound on $I_{\mu}$ by some function
$I:[0,1]\rightarrow\left[ 0,\infty\right)  $ which is not
identically zero.

The modulus of the gradient of a Lipschitz function $f$ on $\Omega$ \ (briefly
$f\in Lip(\Omega))$ is defined by
\[
\left|  \nabla f(x)\right|  =\limsup_{d(x,y)\rightarrow0}\frac{|f(x)-f(y)|}%
{d(x,y)}.
\]

The equivalence between isoperimetric inequalities and Poincar\'{e}
inequalities was obtained by Maz'ya. Maz'ya's method (see
\cite{Maz}, \cite{MM2} and \cite{CMP}) shows that given
$X=X(\Omega)$ a rearrangement invariant space\footnote{i.e. such
that if $f$ and $g$ have the same distribution function then
$\left\| f\right\| _{X}=\left\| g\right\|  _{X}$ (see Section
\ref{ri} below).}, the inequality
\begin{equation}
\left\|  f-\int_{\Omega}fd\mu\right\|  _{X}\leq c\left\|  \left|  \nabla
f\right|  \right\|  _{L^{1}},\text{ }f\in Lip(\Omega), \label{a3}%
\end{equation}
holds, if and only if, there exists a constant $c=c(\Omega)>0$ such that for
all Borel sets $A\subset \Omega,$%
\begin{equation}
\min\left(  \phi_{X}(\mu(A)),\phi_{X}(1-\mu(A))\right)  \leq c\mu^{+}(A),
\label{a4}%
\end{equation}
where $\phi_{X}(t)$ is the fundamental function\footnote{We can
assume with no loss of generality that $\phi_{X}$ is
concave.} of $X:$%
\[
\phi_{X}(t)=\left\|  \chi_{A}\right\|  _{X},\text{ with }\mu(A)=t.
\]
Motivated by this fact, we will say $\left(  \Omega,d,\mu\right)  $ admits a
concave isoperimetric estimator if there exists a function $I:[0,1]\rightarrow
\left[  0,\infty\right)  $ continuous, concave, increasing on $(0,1/2)$,
symmetric about the point $1/2$, $I(0)=0$ and $I(t)>0$ on $(0,1)$, such that%
\[
I_{\mu}(t)\geq I(t),\text{ }0\leq t\leq1.
\]

In their recent work M. Milman and J. Mart\'{\i}n\footnote{J. Martin
is grateful to professor Mario Milman for introducing him, around
2005, in the study of the conexion between rearrangements and
isoperimetry.}  (see \cite{mmadv}, \cite{mmast}) proved
 that $\left(  \Omega,d,\mu\right)  $ admits a concave isoperimetric estimator $I,$ if, and only if, the
following symmetrization inequality holds%
\begin{equation}
f_{\mu}^{\ast\ast}(t)-f_{\mu}^{\ast}(t)\leq\frac{t}{I(t)}\left|  \nabla
f\right|  _{\mu}^{\ast\ast}(t),\text{ (}f\in Lip(\Omega)) \label{int4}%
\end{equation}
where $f_{\mu}^{\ast\ast}(t)=\frac{1}{t}\int_{0}^{t}f_{\mu}^{\ast}(s)ds,$ and
$f_{\mu}^{\ast}$ is the non increasing rearrangement of $f$ with respect to
the measure $\mu$. If we apply a rearrangement invariant function norm $X\ $on
$\Omega$ (see Sections \ref{rea} and \ref{ri} below) to (\ref{int4}) we obtain
Sobolev-Poincar\'{e} type estimates of the form\footnote{The spaces $\bar{X}$
are defined in Section \ref{ri} below.}%
\begin{equation}
\left\|  \left(  f_{\mu}^{\ast\ast}(t)-f_{\mu }^{\ast}(t)\right)
\frac{I(t)}{t}\right\|  _{\bar{X}}\leq\left\|  \left|
\nabla f\right|  _{\mu}^{\ast\ast}\right\|  _{\bar{X}}. \label{int5'}%
\end{equation}

\begin{example}
(see \cite{mmp}, \cite{mmgaus}) Let $\Omega\subset \mathbb{R}^{n}$
be a Lipschitz  domain of measure $1$, $X=L^{p}\left( \Omega\right)
,$ $1\leq p\leq n,$ and $p^{\ast}$ be the usual Sobolev exponent
defined by $\frac{1}{p^{\ast}}=\frac{1}{p}-\frac{1}{n},$
then\footnote{Here the symbol $f\simeq g$ indicates the existence of
a universal constant $c>0$ (independent of all parameters involved)
such that $(1/c)f\leq g\leq c\,f$. Likewise the symbol $f\preceq g$
will mean that there exists a universal constant $c>0$ (independent
of all parameters involved) such that $f\leq c\,g$.}
\begin{equation}
\left\|  \left(  f^{\ast\ast}(t)-f^{\ast}(t)\right)
\frac{I(t)}{t}\right\| _{L^{p}}\simeq\left\|  \left(
f^{\ast\ast}(t)-f^{\ast}(t)\right)  \right\|
_{L^{p^{\ast},p}}, \label{tarde}%
\end{equation}
 follows from the fact that the
isoperimetric profile is equivalent to
$I(t)=c_{n}\min(t,1-t)^{1-1/n},$ and Hardy's inequality (here
$L^{p^{\ast},p}$ is a Lorentz space (see Section \ref{preliminar}
below)). In case that we consider $\mathbb{R}^{n}$ with Gaussian
measure $\gamma_{n}$, and let $X=L^{p},$ $1\leq p<\infty,$ then
(compare with \cite{gro}, \cite{fei}), since
$I_{(\mathbb{R}^{n},d,\gamma_{n})}(t)\simeq t(\log1/t)^{1/2}$ for
$t$ near
zero, we have%
\begin{equation}
\left\|  \left(
f_{\gamma_{n}}^{\ast\ast}(t)-f_{\gamma_{n}}^{\ast}(t)\right)
\frac{I(t)}{t}\right\|  _{L^{p}}\simeq\left\|  \left(  f_{\gamma_{n}}%
^{\ast\ast}(t)-f_{\gamma_{n}}^{\ast}(t)\right)  \right\|
_{L^{p}(Log)^{p/2}},
\label{tarde1}%
\end{equation}
where $L^{p}(logL)^{p/2}$ is a Lorentz-Zygmund space (see Section
\ref{preliminar}).
\end{example}

In this fashion in \cite{mmadv}, \cite{mmast}, \cite{mmp} and
\cite{mmgaus}, M. Milman in collaboration with the first author were
able to provide a unified framework to study the classical Sobolev
inequalities and logarithmic Sobolev inequalities, moreover
embeddings (\ref{int5'}) turn out to be the best possible in all the
classical cases. However the method used in the proof of the
previous results cannot be applied with probability measures with
heavy tails, since isoperimetric estimators of such measures are non
concave. Let us illustrate this phenomenon with some examples (see
\cite[Propositions 4.3 and 4.4]{CGGR} for examples \ref{ex1} and
\ref{ex2} and \cite{Mil} for example \ref{ex3}).

\begin{example}
\label{ex1}($\alpha-$Cauchy type law). Let $\alpha>0.$ Consider the
probability measure space $(\mathbb{R}^{n},d,\mu)$ where $d$ is the
Euclidean distance and $\mu$ is
defined by $d\mu(x)=V^{-(n+\alpha)}dx$ with $V:$ $\mathbb{R}^{n}%
\rightarrow(0,\infty)$ convex. Then there exits $C>0$ such that
for any measurable set $A\subset\mathbb{R}^{n}$%
\[
\mu^{+}(A)\geq C\min\left(  \mu(A),1-\mu(A)\right)  ^{1+1/\alpha}.
\]
\end{example}

\begin{example}
\label{ex2}(Extended $p$-sub-exponential law). Let $p\in (0,1)$.
Consider the probability measure on $\mathbb{R}^{n}$ defined by
$d\mu(x)=\left( 1/Z_{p}\right)  e^{-V^{p}(x)}dx$ for some positive
convex function $V:\mathbb{R}^{n}\to (0,\infty)$, then there exits
$C>0$
such that for any measurable set $A\subset\mathbb{R}^{n}$%
\[
\mu^{+}(A)\geq C\min\left(  \mu(A),1-\mu(A)\right)  \left(  \log\frac{1}%
{\min\left(  \mu(A),1-\mu(A)\right)  }\right)  ^{1-1/p}.
\]
\end{example}

\begin{example}
\label{ex3} Let $\left(  M^{n},g,\mu\right)  $ be a $n-$dimensional weighted
Riemannian manifold ($n\geq2)$ that satisfies the $CD(0,N)$ curvature
condition with $N<0.$ Then for every Borel set $A\subset\left(  M^{n}%
,g\right)  $
\[
\mu^{+}(A)\geq C\min\left(  \mu(A),1-\mu(A)\right)  ^{-1/N}.
\]
\end{example}

Motivated by these examples, we will say $\left(  \Omega,d,\mu\right)  $
admits a convex isoperimetric estimator if there exists a function
$I:[0,1]\rightarrow\left[  0,\infty\right)  $ continuous, convex, increasing
on $(0,1/2)$, symmetric about the point $1/2$, $I(0)=0$ and $I(t)>0$ on
$(0,1)$, such that%
\[
I_{\mu}(t)\geq I(t),\text{ }0\leq t\leq1.
\]

The purpose of this paper is to obtain symmetrization inequalities
on probability metric spaces that admit a convex isoperimetric
estimator which incorporate in their formulation the isoperimetric
estimator and that can be applied to provide a unified treatment of
sharp Sobolev-Poincar\'{e} and Nash type
inequalities. Notice that if $I$ is a convex isoperimetric estimator, then%
\[
I(t) \preceq \min\left(  t,1-t\right) 
\]
therefore (unless $I(t)\simeq\min\left(  t,1-t\right)  ),$ the Poincar\'{e}
inequality
\[
\left\|  f-\int_{\Omega}fd\mu\right\|  _{L^{1}}\leq c\left\|  \left|  \nabla
f\right|  \right\|  _{L^{1}},f\in Lip(\Omega),
\]
never holds, that means from $\left|  \nabla f\right|  \in L^{1}$ we
cannot deduce that $f\in L^{1},$ hence a symmetrization inequality
like (\ref{int4}) will not be possible  since $f_{\mu}^{\ast\ast}$
is defined if, and only if, $f\in L^1$.

The paper is organized as follows. In Section \ref{preliminar}, we
introduce the notation and the standard assumptions used in the
paper. This Section also contains the basic background from the
theory of  rearrangement invariant spaces that we will need. In
Section \ref{simetr} we obtain symmetrization inequalities  which
incorporate in their formulation the isoperimetric convex
estimator.\ In Section \ref{discussion} we use the symmetrization
inequalities to derive Sobolev-Poincar\'{e} and Nash type
inequalities. Finally in Section \ref{five} we study in detail
Examples \ref{ex1}, \ref{ex2} and \ref{ex3}.

\section{Preliminaries\label{preliminar}}

We recall briefly the basic definitions and conventions we use from
the theory of rearrangement-invariant (r.i.) spaces and refer the
reader to \cite{BS}, \cite{KPS} and \cite{Rak} for a complete
treatment.

We shall consider a connected measure metric spaces $\left(
\Omega,d,\mu\right)  $ equipped with a  separable, non-atomic,
probability Borel measure $\mu.$  Let $\mathcal{M}(\Omega)$ be the
set of all extended real-valued measurable functions on $\Omega$. By
$\mathcal{M}_0(\Omega)$ we denote the class of functions in
$\mathcal{M}(\Omega)$ that are finite $\mu-$ a.e.

\subsection{Rearrangements \label{rea}}

For $u\in\mathcal{M}_0(\Omega),$ the distribution
function\footnote{Note that this notation is somewhat
unconventional. In the literature it is common to denote the
distribution function of $\left|  u\right|  $ by $\mu_{u},$ while
here it is denoted by $\mu_{|u|}$ since we need to distinguish
between the rearrangements of $u$ and $\left| u\right| .$ } of $u$
is given by
\[
\mu_{u}(t)=\mu\{x\in{\Omega}:u(x)>t\}\text{ \ \ \ \ }(t\in\mathbb{R}).
\]
The \textbf{decreasing rearrangement} of a function $u$ is the
right-continuous non-increasing function from $[0,1)$ into $[0,\infty)$ which
is equimeasurable with $u,$ i.e.
\[
\mu_{\left|  u\right|  }(t)=\mu\{x\in{\Omega}:\left|  u(x)\right|
>t\}=m\left\{  s\in(0,1):u_{\mu}^{\ast}(s)>t\right\}  \text{ , \ }%
t\in\mathbb{R}%
\]
(where $m$ denotes the Lebesgue measure on $(0,1)$), and can be defined by the
formula%
\[
u_{\mu}^{\ast}(s)=\inf\{t\geq0:\mu_{\left|  u\right|  }(t)\leq s\}.
\]
The \textbf{signed decreasing rearrangement} of $f$ is defined by
\[
u_{\mu}^{\bigstar}(t)=\inf\left\{  s\in\mathbb{R}:\mu\{x\in\Omega:\mu
_{u}(x)>s\}\leq t\right\}  ,
\]
It follows readily from the definition that
\begin{equation}
u_{\mu}^{\bigstar}(0^{+})=ess\sup u\text{ \ \ and \ \ \ }u_{\mu}^{\bigstar
}(\infty)=ess\inf u. \label{signed}%
\end{equation}

The maximal function $u_{\mu}^{\ast\ast}\ $of $u_{\mu}^{\ast}$ is defined by
\[
u_{\mu}^{\ast\ast}(t)=\frac{1}{t}\int_{0}^{t}u_{\mu}^{\ast}(s)ds=\frac{1}%
{t}\sup\left\{  \int_{E}u(s)d\mu:\mu(E)=t\right\}  .
\]
This operation is subadditive, i.e.
\begin{equation}
\left(  u+v\right)  _{\mu}^{\ast\ast}(s)\leq u_{\mu}^{\ast\ast}(s)+v_{\mu
}^{\ast\ast}(s). \label{a2}%
\end{equation}
Moreover, since $u_{\mu}^{\ast}$ is decreasing, $u_{\mu}^{\ast\ast}$ is also
decreasing and $u_{\mu}^{\ast}\leq u_{\mu}^{\ast\ast}$.

When no confusion ensues, because the measure is clear from the context, or we
are dealing with Lebesgue measure, we may simply write $u^{\ast}$ and
$u^{\ast\ast}$.

\begin{definition}
Let $f\in\mathcal{M}_0(\Omega)$. We say
that $m(f)$ is a median value if
\[
\mu\{f\geq m(f)\}\geq1/2;\text{ and }\mu\{f\leq m(f)\}\geq1/2.
\]
\end{definition}
It is easy to see (see for example \cite{mmadv}) that
$f_{\mu}^{\bigstar}(1/2)$ is a median of $f.$ Moreover, if $f$ has
$0$ median and $f_{\mu}^{\bigstar}$ is continuous then $f_{\mu
}^{\bigstar}(1/2)=0.$

\subsection{Rearrangement invariant spaces\label{ri}}

We say that a Banach function space $X=X({\Omega})$ on $({\Omega},d,\mu)$ is
rearrangement-invariant (r.i.) space, if
\begin{equation}
f\in X\text{ and }g\text{ is a  }\mu-\text{measurable such that
}f_{\mu}^{\ast }=g_{\mu}^{\ast}\Rightarrow g\in X\text{ and }\Vert
f\Vert_{X}=\Vert
g\Vert_{X}. \label{riri}%
\end{equation}

If, in the definition of a norm, the triangle inequality is weakened to the
requirement that for some constant $C_{X}$, $\left\|  x+y\right\|  _{X}\leq
C_{X}(\left\|  x\right\|  _{X}+\left\|  y\right\|  _{X})$ holds for all $x$
and $y$, then we have a quasi-norm. A complete quasi-normed space is called a
quasi-Banach space. We will say that $X$ is a quasi Banach
rearrangement-invariant (q.r.i.) space if (\ref{riri}) holds.

If $X$ is a r.i space, since $\mu(\Omega)=1,$ for any r.i. space $X({\Omega})$
we have%
\begin{equation}
L^{\infty}(\Omega)\subset X(\Omega)\subset L^{1}(\Omega), \label{nuevadeli}%
\end{equation}
with continuous embeddings.

A r.i. space $X({\Omega})$ can be represented by a r.i. space on the interval
$(0,1),$ with Lebesgue measure, $\bar{X}=\bar{X}(0,1),$ such that%
\[
\Vert f\Vert_{X}=\Vert f_{\mu}^{\ast}\Vert_{\bar{X}},
\]
for every $f\in X.$ A characterization of the norm $\Vert\cdot\Vert_{\bar{X}}$
is available (see \cite[Theorem 4.10 and subsequent remarks]{BS}).

A useful property of r.i. spaces states that if
\[
\int_{0}^{r}f_{\mu}^{\ast}(s)ds\leq\int_{0}^{r}g_{\mu}^{\ast}(s)ds,\text{
\ holds for all\ }r>0,
\]
then, for any Banach r.i. space $X=X({\Omega}),$%
\[
\left\|  f\right\|  _{X}\leq\left\|  g\right\|  _{X}.
\]

Classically conditions on r.i. spaces can be formulated in terms of the
boundedness of the Hardy operators defined by
\[
Pf(t)=\frac{1}{t}\int_{0}^{t}f(s)ds;\text{ \ \ \ }Qf(t)=\int_{t}^{1  }f(s)\frac{ds}{s}.
\]
The boundedness of these operators on r.i. spaces can be best described in
terms of the so called \textbf{Boyd indices}\footnote{Introduced by D.W. Boyd
in \cite{boyd}.} defined by
\begin{equation}
\bar{\alpha}_{X}=\inf\limits_{s>1}\dfrac{\ln h_{X}(s)}{\ln s}\text{ \ \ and
\ \ }\underline{\alpha}_{X}=\sup\limits_{s<1}\dfrac{\ln h_{X}(s)}{\ln s},
\label{boy}%
\end{equation}
where $h_{X}(s)$ denotes the norm of the compression/dilation operator $E_{s}$
on $\bar{X}$, defined for $s>0,$ by
\[
E_{s}f(t)=\left\{
\begin{array}
[c]{ll}%
f^{\ast}(\frac{t}{s}) & 0<t<s,\\
0 & s<t<1.
\end{array}
\right.
\]
It is well known that
\begin{equation}%
\begin{array}
[c]{c}%
P\text{ is bounded on }\bar{X}\text{ }\Leftrightarrow\overline{\alpha}%
_{X}<1,\\
Q\text{ is bounded on }\bar{X}\text{ }\Leftrightarrow\underline{\alpha}_{X}>0.
\end{array}
\label{alcance}%
\end{equation}

The next Lemma will be useful in what follows (see \cite[Lemma
1]{Mar1}).

\begin{lemma}
\label{indices}Let $Y$ be a r.i. space on $\left(  0,1\right)  .$
Let $\phi_{Y}$ be its fundamental function. Assume that
$\phi_{Y}(0)=0.$ Then
\end{lemma}

\begin{enumerate}
\item  If $\overline{\alpha}_{Y}<1$, then for every $\overline{\alpha}%
_{Y}<\gamma<1$ the function $\mathbb{\phi}_{Y}(s)/s^{\gamma}$ is
almost decreasing (i. e. $\exists c>0$ s.t.
$\mathbb{\phi}_{Y}(s)/s^{\gamma}\leq
c\mathbb{\phi}_{Y}(t)/t^{\gamma}$ whenever $t\leq s).$

\item  If $\underline{\alpha}_{Y}>0,$ then for every $0<\gamma<\underline
{\alpha}_{Y}$ the function $\mathbb{\phi}_{Y}(s)/s^{\gamma}$ is
almost increasing (i. e. $\exists c>0$ s.t.
$\mathbb{\phi}_{Y}(s)/s^{\gamma}\leq
c\mathbb{\phi}_{Y}(t)/t^{\gamma}$ whenever $t\geq s).$

\item  If $\underline{\alpha}_{Y}>0,$ there exists a concave function
$\mathbb{\hat{\phi}}_{Y}$ and constant $c>0$ such that
\[
\mathbb{\hat{\phi}}_{Y}(t)\simeq\mathbb{\phi}_{Y}(t)\text{ \ \ and \ }%
c^{-1}\mathbb{\phi}_{Y}(t)/t\leq\frac{\partial}{\partial t}\mathbb{\hat{\phi}%
}_{Y}(t)\leq c\mathbb{\phi}_{Y}(t)/t.
\]
\end{enumerate}

Associated with an r.i. space $X$ there are some useful Lorentz and
Marcinkiewicz spaces, namely the Lorentz and Marcinkiewicz spaces
defined by the quasi-norms
\[
\left\|  f\right\|  _{M(X)}=\sup_{t}f_{\mu}^{\ast}(t)\phi_{X}(t),\text{
\ \ }\left\|  f\right\|  _{\Lambda(X)}=\int_{0}^{1}f_{\mu}^{\ast}%
(t)\frac{\partial}{\partial t}\phi_{X}(t).
\]
Notice that
\[
\phi_{M(X)}(t)=\phi_{\Lambda(X)}(t)=\phi_{X}(t),
\]
and that
\begin{equation}
\Lambda(X)\subset X\subset M(X). \label{tango}%
\end{equation}

\subsubsection{Examples}

\textbf{Classical Lorentz spaces}: The spaces $L^{p,q}\left(  \Omega\right)  $
are defined by the function quasi-norm%
\[
\left\|  f\right\|  _{p,q}=\left(  \int_{0}^{1}\left(  s^{1/p}f_{\mu}^{\ast
}(s)\right)  ^{q}\frac{ds}{s}\right)  ^{1/q},
\]
when $0<p,q<\infty,$ and
\[
\left\|  f\right\|
_{p,\infty}=\sup_{0<t<1}s^{1/p}f_{\mu}^{\ast}(s),
\]
when $q=\infty.$ Note that $\left\|  f\right\|  _{p,p}=\left\|
f\right\| _{p}.$ (We use the standard convention $\left\|  f\right\|
_{\infty,\infty
}=\left\|  f\right\|  _{\infty}).$\\

\noindent
\textbf{Lorentz-Zygmund}-\textbf{spaces. }Let $1\leq p,q\leq\infty$ and
$\alpha\in\mathbb{R}.$ The spaces $L^{p,q}(\log L)^{\alpha}$ are defined by
the function quasi-norm%
\[
\left\|  f\right\|  _{p,q,\alpha}=\left(  \int_{0}^{1}\left(  s^{1/p}%
(1+\ln\frac{1}{t})^{\alpha}f_{\mu}^{\ast}(s)\right)  ^{q}\frac{ds}{s}\right)
^{1/q}.
\]
\textbf{Weighted q.r.i-spaces}: Given $X$ a r.i. space on $\Omega$
and $w$ a weight (i.e a positive measurable function), we define
\[
X(w)=\left\{  f:\left\|  f\right\|  _{X(w)}=\left\|  f_{\mu}^{\ast}w\right\|
_{\bar{X}}<\infty\right\}  .
\]
It is easy to see that $X(w)$ is a q.r.i-space. For example if $X=L^{q}%
(\Omega)$ and $w(s)=s^{q/p-1}$ then
\[
L^{q}(w)=L^{p,q}\left(  \Omega\right)  .
\]

\section{Symmetrization and Isoperimetry\label{simetr}}

We will assume in what follows that $\left( \Omega,d,\mu\right)$ is
a connected measure metric spaces equipped with a  with a separable,
non-atomic, probability Borel measure $\mu$ which admits a convex
isoperimetric estimator.

In order to balance generality with power and simplicity, we will assume
throughout the paper that our spaces satisfy the following:

\begin{condition}
\label{cond2.1} We assume that $\Omega$ is such that for every $f\in
Lip(\Omega)$ and every $c\in\mathbb{R}$ we have that $\left|  \nabla
f(x)\right|  =0,$ a.e. on the set $\{x:f(x)=c\}.$
\end{condition}

\begin{theorem}
\label{tm1}Let $I:[0,1]\rightarrow\lbrack0,\infty)$ be a convex isoperimetric
estimator$.$ The following statements are equivalent:

\begin{enumerate}
\item  Isoperimetric inequality: for all Borel sets $A\subset\Omega,$%

\begin{equation}
\mu^{+}(A)\geq I(\mu(A)). \label{eq2.2}%
\end{equation}

\item  Ledoux's inequality (cf \cite{Le}): for all $f\in Lip(\Omega),$%

\begin{equation}
\int_{-\infty}^{\infty}I(\mu_{f}(s))\leq\int_{\Omega}\left|  \nabla
f(x)\right|  d\mu. \label{eq2.3}%
\end{equation}

\item  For all function $f\in Lip(\Omega),$ $f_{\mu}^{\bigstar}$ is locally
absolutely continuous, and
\begin{equation}
\int_{0}^{t}((-f_{\mu}^{\bigstar})^{\prime}I(s))^{\ast}ds\leq\int_{0}%
^{t}\left|  \nabla f\right|  _{\mu}^{\ast}(s)ds. \label{reafun}%
\end{equation}
(The second rearrangement on the left hand side is with respect to the
Lebesgue measure).

\item  Bobkov's inequality (cf \cite{bob}): For all $f\in Lip(\Omega)$ bounded
with $m(f)=0$, and for all $s>0$%
\begin{equation}
\int_{\Omega}\left|  f(x)\right|  d\mu\leq\beta_{1}(s)\int_{\Omega}\left|
\nabla f(x)\right|  d\mu+s\,Osc_{\mu}(f), \label{bob}%
\end{equation}
where $Osc_{\mu}(f)=ess\sup f-ess\inf f,$ and $\beta_{1}(s)=\displaystyle\sup_{s<t\leq
1/2}\frac{t-s}{I(t)}.$
\end{enumerate}
\end{theorem}

\begin{proof}
$1)\rightarrow2)$ By the co-area inequality applied to $f$ (cf. \cite[Lemma
3.1]{BH}), and the isoperimetric inequality (\ref{eq2.2}), it follows that
\begin{align*}
\int_{\Omega}\left|  \nabla f(x)\right|  d\mu &  \geq\int_{-\infty}^{\infty
}\mu^{+}(\{f>s\};\Omega)ds\\
&  \geq\int_{0}^{\infty}I(\mu_{f}(s))ds\text{ }.
\end{align*}

$2)\rightarrow3)$ Let $-\infty<t_{1}<t_{2}<\infty.$ The smooth truncations of
$f$ are defined by
\[
f_{t_{1}}^{t_{2}}(x)=\left\{
\begin{array}
[c]{ll}%
t_{2}-t_{1} & \text{if }f(x)\geq t_{2},\\
f(x)-t_{1} & \text{if }t_{1}<f(x)<t_{2},\\
0 & \text{if }f(x)\leq t_{1}.
\end{array}
\right.
\]

\noindent Obviously, $f_{t_{1}}^{t_{2}}\in Lip(\Omega),$ thus by
(\ref{eq2.2}), we get
\[
\int_{-\infty}^{\infty}I(\mu_{f_{t_{1}}^{t_{2}}}(s))ds\leq\int_{\Omega}\left|
\nabla f_{t_{1}}^{t_{2}}(x)\right|  d\mu.
\]
\newline By condition \ref{cond2.1}
\[
\left|  \nabla f_{t_{1}}^{t_{2}}\right|  =\left|  \nabla f\right|
_{\chi_{\{t_{1}<f<t_{2}\}}},
\]
and moreover,
\[
\int_{-\infty}^{\infty}I(\mu_{f_{t_{1}}^{t_{2}}}(s))ds=\int_{t_{1}}^{t_{2}%
}I(\mu_{f_{t_{1}}^{t_{2}}}(s))ds.
\]
\noindent Observe that, $t_{1}<z<t_{2},$
\[
\mu\{f\geq t_{2}\}\leq\mu_{f_{t_{1}}^{t_{2}}}(z)\leq\mu\{f>t_{1}\}.
\]
Consequently, by the properties of $I,$ we have
\begin{equation}
\int_{t_{1}}^{t_{2}}I(\mu_{f_{t_{1}}^{t_{2}}}(z))dz\geq(t_{2}-t_{1}%
)\min\{I(\mu\{f\geq t_{2}\}),I(\mu\{f>t_{1}\})\}. \label{prim}%
\end{equation}
\noindent Let us see that $f_{\mu}^{\bigstar}$ is locally absolutely
continuous. Indeed, for $s>0$ and $h>0,$ pick $t_{1}=f_{\mu}^{\bigstar}(s+h),$
$t_{2}=f_{\mu}^{\bigstar}(s),$ then
\begin{equation}
s\leq\mu\{f(x)\geq f_{\mu}^{\bigstar}(s)\}\leq\mu_{f_{t_{1}}^{t_{2}}}%
(s)\leq\mu\{f(x)>f_{\mu}^{\bigstar}(s+h)\}\leq s+h. \label{prim1}%
\end{equation}
Combining (\ref{prim}) and (\ref{prim1}) we have,
\begin{equation}
(f_{\mu}^{\bigstar}(s)-f_{\mu}^{\bigstar}(s+h))\min\{I(s+h),I(s)\}\leq
\int_{\left\{  f_{\mu}^{\bigstar}(s)<f<f_{\mu}^{\bigstar}(s+h)\right\}
}\left|  \nabla f(x)\right|  d\mu\label{prim2}%
\end{equation}
\noindent which implies that $f_{\mu}^{\bigstar}$ is locally absolutely
continuous in $\left[  a,b\right]  $ $(0<a<b<1).$ Indeed, for any finite
family of non-overlapping intervals $\left\{  (a_{k},b_{k})\right\}
_{k=1}^{r},$ with $(a_{k},b_{k})\subset\left[  a,b\right]  ,$ and $\sum
_{k=1}^{r}(b_{k}-a_{k})\leq\delta,$ we have
\[
\mu\{\bigcup_{k=1}^{r}\{f_{\mu}^{\bigstar}(b_{k})<f<f_{\mu}^{\bigstar}%
(a_{k})\}\}=\sum_{k=1}^{r}\mu\{f_{\mu}^{\bigstar}(b_{k})<f<f_{\mu}^{\bigstar
}(a_{k})\}\leq\sum_{k=1}^{r}(b_{k}-a_{k})\leq\delta.
\]
\noindent therefore, combining this fact with (\ref{prim2}), we have
\begin{align*}
\sum_{k=1}^{r}(f_{\mu}^{\bigstar}(a_{k})-f_{\mu}^{\bigstar}(b_{k}%
))\min\{I(a),I(b)\}  &  \leq\sum_{k=1}^{r}(f_{\mu}^{\bigstar}(a_{k})-f_{\mu
}^{\bigstar}(b_{k}))\min\{I(a_{k}),I(b_{k})\}\\
&  \leq\sum_{k=1}^{r}\int_{\{f_{\mu}^{\bigstar}(b_{k})<f<f_{\mu}^{\bigstar
}(a_{k})\}}\left|  \nabla f(x)\right|  d\mu\\
&  =\int_{\cup_{k=1}^{r}\{f_{\mu}^{\bigstar}(b_{k})<f<f_{\mu}^{\bigstar}%
(a_{k})\}}\left|  \nabla f(x)\right|  d\mu\\
&  \leq\int_{0}^{\sum_{k=1}^{r}(b_{k}-a_{k})}\left|  \nabla f\right|  _{\mu
}^{\ast}(t)dt\\
&  \leq\int_{0}^{\delta}\left|  \nabla f\right|  _{\mu}^{\ast}(t)dt,
\end{align*}
and the local absolute continuity follows.\newline \noindent Now, using
(\ref{prim2}) we get,
\begin{align*}
\dfrac{(f_{\mu}^{\bigstar}(s)-f_{\mu}^{\bigstar}(s+h))}{h}\min(I(s+h),I(s))
&  \leq\int_{\{f_{\mu}^{\bigstar}(s+h)<f<f_{\mu}^{\bigstar}(s)\}}\left|
\nabla f(x)\right|  d\mu\\
&  \leq\dfrac{1}{h}\int_{\{f_{\mu}^{\bigstar}(s+h)<f\leq f_{\mu}^{\bigstar
}(s)\}}\left|  \nabla f(x)\right|  d\mu.
\end{align*}
\noindent Letting $h\rightarrow0,$%
\[
(-f_{\mu}^{\bigstar})^{\prime}(s)I(s)\leq\frac{\partial}{\partial s}%
\int_{\{f>f_{\mu}^{\bigstar}(s)\}}\left|  \nabla f(x)\right|  d\mu.
\]

Let us consider a finite family of intervals $\left(  a_{i},b_{i}\right)  ,$
$i=1,\ldots,m$, with $0<a_{1}<b_{1}\leq a_{2}<b_{2}\leq\cdots\leq a_{m}%
<b_{m}<1,$ then
\begin{align*}
\int_{\cup_{1\leq i\leq m}(a_{i},b_{i})}\left(  -f_{\mu}^{\bigstar}\right)
^{{\prime}}(s)I(s)ds  &  \leq\int_{\cup_{1\leq i\leq m}(a_{i},b_{i})}\left(
\frac{\partial}{\partial s}\int_{\left\{  \left|  f\right|  >f_{\mu}%
^{\bigstar}(s)\right\}  }\left|  \nabla f(x)\right|  d\mu(x)\right)  ds\\
&  =\sum_{i=1}^{m}\int_{\left\{  f_{\mu}^{\bigstar}(b_{i})<\left|  f\right|
\leq f_{\mu}^{\bigstar}(a_{i})\right\}  }\left|  \nabla f(x)\right|  d\mu(x)\\
&  =\sum_{i=1}^{m}\int_{\left\{  f_{\mu}^{\bigstar}(b_{i})<\left|  f\right|
<f_{\mu}^{\bigstar}(a_{i})\right\}  }\left|  \nabla f(x)\right|  d\mu(x)\text{
(by condition \ref{cond2.1}))}\\
&  =\int_{\cup_{1\leq i\leq m}\left\{  f_{\mu}^{\bigstar}(b_{i})<\left|
f\right|  <f_{\mu}^{\bigstar}(a_{i})\right\}  }\left|  \nabla f(x)\right|
d\mu(x)\\
&  \leq\int_{0}^{\sum_{i=1}^{m}\left(  b_{i}-a_{i}\right)  }\left|  \nabla
f\right|  _{\mu}^{\ast}(s)ds.
\end{align*}
Now by a routine limiting process we can show that for any
measurable set $E\subset$ $(0,1),$ with Lebesgue measure equal to
$t,$ we have
\[
\int_{E}(-f_{\mu}^{\bigstar})^{\prime}(s)I(s)ds\leq\int_{0}^{|E|}\left|
\nabla f\right|  _{\mu}^{\ast}(s)ds.
\]
Therefore
\begin{equation}
\int_{0}^{t}((-f_{\mu}^{\bigstar})^{\prime}(\cdot)I(\cdot))^{\ast}%
(s)ds\leq\int_{0}^{t}\left(  \left|  \nabla f\right|  _{\mu}^{\ast}%
(\cdot)\right)  ^{\ast}(s)ds, \label{aa}%
\end{equation}
where the second rearrangement is with respect to the Lebesgue
measure. Now, since $\left|  \nabla f\right|  _{\mu}^{\ast}(s)$ is
decreasing, we have
\[
\left(  \left|  \nabla f\right|  _{\mu}^{\ast}(\cdot)\right)  ^{\ast
}(s)=\left|  \nabla f\right|  _{\mu}^{\ast}(s),
\]
and thus (\ref{aa}) yields
\[
\int_{0}^{t}((-f_{\mu}^{\bigstar})^{\prime}(\cdot)I(\cdot))^{\ast}%
(s)ds\leq\int_{0}^{t}\left|  \nabla f\right|  _{\mu}^{\ast}(s)ds.
\]

$3)\rightarrow4)$ Assume first that $f\in Lip(\Omega)$ is positive,
bounded with $m(f)=0$. By 3) we have that
$f_{\mu}^{\bigstar}=f_{\mu}^{\ast}$ (since $f\geq0)$ is locally
absolutely continuous and $f_{\mu}^{\ast}(1/2)=0$ (since $m(f)=0$).
Let $0<s<z\leq1/2,$ then
\begin{align*}
\int_{\Omega}\left|  f(x)\right|  d\mu &  =\int_{0}^{1/2}f_{\mu}^{\ast
}(z)dz=\int_{0}^{1/2}\int_{z}^{1/2}(-f_{\mu}^{\ast})^{\prime}(x)dxdz=\\
&  =\int_{0}^{1/2}z(-f_{\mu}^{\ast})^{\prime}(z)dz-s\int_{0}^{1/2}(-f_{\mu
}^{\ast})^{\prime}(z)dz+s\int_{0}^{1/2}(-f_{\mu}^{\ast})^{\prime}(z)dz\\
&  =\int_{0}^{1/2}\frac{z-s}{I(z)}(-f_{\mu}^{\ast})^{\prime}(z)I(z)dz+s\int
_{0}^{1/2}(-f_{\mu}^{\ast})^{\prime}(z)dz\\
&  \leq\sup_{s<z\leq1/2}\frac{z-s}{I(z)}\int_{0}^{1/2}(-f_{\mu}^{\ast
})^{\prime}(z)I(z)dz+s\int_{0}^{1/2}(-f_{\mu}^{\ast})^{\prime}(z)dz\\
&  \leq\beta_{1}(s)\int_{0}^{1/2}(-f_{\mu}^{\ast})^{\prime}(z)I(z)dz+s\int
_{0}^{1/2}(-f_{\mu}^{\ast})^{\prime}(z)dz.
\end{align*}

Since
\[
s\int_{0}^{1/2}(-f_{\mu}^{\ast})^{\prime}(z)=s(f_{\mu}^{\ast}(0^{+})-f_{\mu
}^{\ast}(1/2))\leq s\,Osc_{\mu}(f),
\]
we get
\begin{align*}
\int_{\Omega}\left|  f(x)\right|  d\mu &  \leq\beta_{1}(s)\int_{0}^{t}%
(-f_{\mu}^{\ast})^{\prime}(z)I(z)dz+s\,Osc_{\mu}(f)\\
&  \leq\beta_{1}(s)\int_{0}^{1/2}\left(  (-f_{\mu}^{\ast})^{\prime}%
(\cdot)I(\cdot)\right)  ^{\ast}(t)dt+s\,Osc_{\mu}(f)\\
&  \leq\beta_{1}(s)\int_{0}^{1/2}\left|  \nabla f\right|  _{\mu}^{\ast
}(t)dt+s\,Osc_{\mu}(f)\text{ \ \ (by (\ref{reafun}))}\\
&  =\beta_{1}(s)\int_{\Omega}\left|  \nabla f(x)\right|
d\mu+s\,Osc_{\mu
}(f)\text{ }%
\end{align*}
In the general case, we follow \cite[Lemma 8.3]{bob}. Apply the
previous argument to $f^{+}=\max(f,0)$ and $f^{-}=\max(-f,0),$ which
are positive, Lipschitz and have median zero, and we obtain.
\[
\int_{\left\{  f>0\right\}  }\left|  f(x)\right|  d\mu\leq\beta_{1}%
(s)\int_{\left\{  f>0\right\}  }\left|  \nabla f(x)\right|
d\mu+s\,Osc_{\mu
}(f^{+})\text{,}%
\]%
\[
\int_{\left\{  f<0\right\}  }\left|  f(x)\right|  d\mu\leq\beta_{1}%
(s)\int_{\left\{  f><0\right\}  }\left|  \nabla f(x)\right|
d\mu+s\,Osc_{\mu
}(f^{-}).\text{ }%
\]
Adding the two inequalities and since
$Osc_{\mu}(f^{-})+Osc_{\mu}(f^{+})\leq Osc_{\mu}(f),$ we get
(\ref{bob}).

$4)\rightarrow1)$ This part was proved in \cite[Lemma 8.3]{bob}, we
include its proof for the sake of completeness. Given a Borel set
$A\subset\Omega$ we may approximate the indicator function
$\chi_{A}$ by functions with finite Lipschitz seminorm (see
\cite{BH}) to derive $\mu(A)\leq\beta_{1}(s)\mu
^{+}(A)+s,$ therefore if $\mu(A)=t$%
\[
t-s\leq\beta_{1}(s)\mu^{+}(A)
\]
thus the optimal choice should be%
\[
I(t)=\sup_{0<s<t}\frac{t-s}{\beta_{1}(s)}.
\]
\end{proof}

\section{Sobolev-Poincar\'{e} and Nash type inequalities\label{discussion}}

The isoperimetric inequality implies weaker Sobolev-Poincar\'{e} and
Nash type inequalities, in what follows we will analyze both.

\subsection{Sobolev-Poincar\'{e} inequalities}

The isoperimetric Hardy operator $Q_{I}$ is the operator defined on
Lebesgue
measurable functions on $(0,1)$ by%
\[
Q_{I}f(t)=\int_{t}^{1/2}f(s)\frac{ds}{I(s)}\text{, \ \ \ \ }0<t<1/2,
\]
where $I$ is a convex isoperimetric estimator. In this section we consider the
possibility of characterizing Sobolev embeddings in terms of the boundedness
of $Q_{I}.$

\begin{lemma}
\label{acot}Let $Y,$ $Z$ be two q.r.i spaces on $(0,1).$ Assume that there is
a constant $C_{0}>0$ such that
\begin{equation}
\left\|  Q_{I}f\right\|  _{Y}\leq C_{0}\left\|  f\right\|  _{Z}.
\label{acotado}%
\end{equation}
Then, there exists a constant $C_{1}>0$ such that
\[
\left\|  \bar{Q}_{I}f\right\|  _{Y}\leq C_{1}\left\|  f\right\|  _{Z},
\]
where $\bar{Q}_{I}$ is the operator defined on Lebesgue measurable functions
on $(0,1)$ by
\[
\bar{Q}_{I}f(t)=\int_{t}^{1/2}f(s)\frac{ds}{I(s)}\text{, \ \ \ \ }0<t<1.
\]
\end{lemma}

\begin{proof}
Since
\begin{align*}
\bar{Q}_{I}f(t)  &  =\chi_{(0,1/2)}(t)\bar{Q}_{I}f(t)+\chi_{(1/2,1)}%
(t)\int_{t}^{1/2}f(s)\frac{ds}{I(s)}\\
&  =\chi_{(0,1/2)}(t)Q_{I}f(t)+\chi_{(1/2,1)}(t)\int_{t}^{1/2}f(s)\frac
{ds}{I(s)},
\end{align*}
it is enough to prove the boundedness of $\chi_{(1/2,1)}(t)\int_{t}%
^{1/2}f(s)\frac{ds}{I(s)}.$

For $t\in(1/2,1),$ we have that%
\begin{align*}
\int_{t}^{1/2}f(s)\frac{ds}{I(s)}  &  =-\int_{1/2}^{t}f(s)\frac{ds}{I(s)}%
=\int_{1/2}^{1-t}f(1-s)\frac{ds}{I(1-s)}\\
&  =-\int_{1-t}^{1/2}f(1-s)\frac{ds}{I(s)}\text{ \ (since }I(s)=I(1-s)).
\end{align*}
Thus%
\begin{align*}
\left\|  \chi_{(1/2,1)}(t)\int_{t}^{1/2}f(s)\frac{ds}{I(s)}\right\|  _{Y}  &
=\left\|  \chi_{(1/2,1)}(t)\int_{1-t}^{1/2}f(1-s)\frac{ds}{I(s)}\right\|
_{Y}\\
&  =\left\|  \chi_{(1/2,1)}(1-t)\int_{t}^{1/2}f(1-s)\frac{ds}{I(s)}\right\|
_{Y}\text{ \ \ (since }\left\|  \cdot\right\|  _{Y}\text{ is r.i)}\\
&  =\left\|  \chi_{(0,1/2)}(t)\int_{t}^{1/2}f(1-s)\frac{ds}{I(s)}\right\|
_{Y}\\
&  \leq C\left\|  \chi_{(0,1/2)}(t)f(1-t)\right\|  _{Z}\\
&  \leq C\left\|  f(t)(\chi_{(0,1/2)}(1-t))\right\|  _{Z}\text{ \ (since
}\left\|  \cdot\right\|  _{\bar{X}}\text{ is r.i)}\\
&  \leq C\left\|  \chi_{(1/2,1)}(t)f(t)\right\|  _{Z}\\
&  \leq C\left\|  f\right\|  _{Z}.
\end{align*}
\end{proof}

\begin{theorem}
\label{t1}Let $Y$ be a q.r.i. space on $(0,1),$ and let $X$ be a r.i. space on
$\Omega.$ Assume that there is a constant $C>0$ such that
\begin{equation}
\left\|  Q_{I}f\right\|  _{Y}\leq C\left\|  f\right\|  _{\bar{X}}
\label{hiphip}%
\end{equation}
then, for all $g\in Lip(\Omega)$ we have that
\[
\inf_{c\in\mathbb{R}}\left\|  \left(  g-c\right)  _{\mu}^{\ast}\right\|
_{Y}\preceq\left\|  \left|  \nabla g\right|  \right\|  _{X}.
\]
\end{theorem}

\begin{proof}
Given $g\in Lip(\Omega)$, by part $2$ of Theorem \ref{tm1}, $g_{\mu}%
^{\bigstar}$ is locally absolutely continuous on $(0,1).$ Thus, for
$t\in(0,1),$ we have that
\begin{align*}
\left|  g_{\mu}^{\bigstar}(t)-g_{\mu}^{\bigstar}(1/2)\right|   &  =\left|
\int_{t}^{1/2}\left(  -g_{\mu}^{\bigstar}\right)  ^{\prime}(s)ds\right|
=\left|  \int_{t}^{1/2}\left(  -g_{\mu}^{\bigstar}\right)  ^{\prime
}(s)I(s)\frac{ds}{I(s)}\right| \\
&  =\left|  \bar{Q}_{I}\left(  \left(  -g_{\mu}^{\bigstar}\right)  ^{\prime
}(\cdot)I(\cdot)\right)  (t)\right|  .\text{\ }%
\end{align*}
Then
\begin{align*}
\left\|  \left|  g_{\mu}^{\bigstar}(t)-g_{\mu}^{\bigstar}(1/2)\right|
\right\|  _{Y}  &  =\left\|  \bar{Q}_{I}\left(  \left(  -g_{\mu}^{\bigstar
}\right)  ^{\prime}(\cdot)I(\cdot)\right)  (t)\right\|  _{Y}\\
&  \preceq\left\|  \left(  -g_{\mu}^{\bigstar}\right)  ^{\prime}(\cdot
)I(\cdot)\right\|  _{\bar{X}}\text{ \ \ (by (\ref{hiphip}) and Lemma
\ref{acot})}\\
&  \preceq\left\|  \left|  \nabla g\right|  \right\|  _{X}\text{ \ (by
(\ref{reafun})).}%
\end{align*}
Therefore%
\begin{align*}
\inf_{c\in\mathbb{R}}\left\|  \left(  g-c\right)  _{\mu}^{\ast}(t)\right\|
_{Y}  &  =\inf_{c\in\mathbb{R}}\left\|  \left(  g-c\right)  _{\mu}^{\bigstar
}(t)\right\|  _{Y}\\
&  \leq\left\|  \left|  g_{\mu}^{\bigstar}(t)-g_{\mu}^{\bigstar}(1/2)\right|
\right\|  _{Y}\\
&  \preceq\left\|  \left|  \nabla g\right|  \right\|  _{X}.
\end{align*}
\end{proof}

\begin{theorem}
\label{t2}Let $X$ be a r.i. space on $\Omega.$ Assume that
$\underline{\alpha }_{X}>0$ or that there is $c>0$ such that the
convex isoperimetric estimator $I,$ satisfies that
\begin{equation}
\int_{t}^{1/2}\frac{ds}{I(s)}\leq c\frac{t}{I(t)},\text{ \ \ }0<t<1/2.
\label{peso}%
\end{equation}
Then, for all $g\in Lip(\Omega)$, we have that
\begin{equation}
\inf_{c\in\mathbb{R}}\left\|  \left(  g-c\right)  _{\mu}^{\ast}(t)\frac
{I(t)}{t}\right\|  _{\bar{X}}\preceq\left\|  \left|  \nabla g\right|
\right\|  _{X}. \label{desii}%
\end{equation}

Moreover, if $Y$ is a q.r.i. space on $(0,1)$ such that%
\begin{equation}
\left\|  Q_{I}f\right\|  _{Y}\preceq\left\|  f\right\|  _{\bar{X}}, \label{hip}%
\end{equation}
then for all $\mu-$measurable function $g$ on $\Omega,$ we have that%
\[
\left\|  g_{\mu}^{\ast}\right\|  _{Y}\preceq\left\|  g_{\mu}^{\ast}%
(t)\frac{I(t)}{t}\right\|  _{\bar{X}}.
\]
In particular, for all $g\in Lip(\Omega)$, we get%
\[
\inf_{c\in\mathbb{R}}\left\|  \left(  g-c\right)  _{\mu}^{\ast}\right\|
_{Y}\preceq\inf_{c\in\mathbb{R}}\left\|  \left(  g-c\right)  _{\mu}^{\ast
}(t)\frac{I(t)}{t}\right\|  _{\bar{X}}\preceq\left\|  \left|  \nabla g\right|
\right\|  _{X}.
\]
\end{theorem}

\begin{proof}
We associate to the r.i. space $\bar{X}$, the weighted q.r.i space $Z$ on
$(0,1)$ which quasi-norm is defined by
\[
\left\|  f\right\|  _{Z}:=\left\|  f^{\ast}(t)\frac{I(t)}{t}\right\|
_{\bar{X}}.
\]
We claim that there is $C>0$ such that
\[
\left\|  Q_{I}f\right\|  _{Z}\leq C\left\|  f\right\|  _{\bar{X}},
\]
and therefore (\ref{desii}) follows by Theorem \ref{t1}.

Case $1$: $\underline{\alpha}_{X}>0:$%

\begin{align*}
\left\|  Q_{I}f\right\|  _{Z}  &  =\left\|  \frac{I(t)}{t}\left(  \int
_{t}^{1/2}f(s)\frac{ds}{I(s)}\right)  ^{\ast}\right\|  _{\bar{X}}\leq\left\|
\frac{I(t)}{t}\left(  \int_{t}^{1/2}\left|  f(s)\right|  \frac{ds}%
{I(s)}\right)  ^{\ast}\right\|  _{\bar{X}}\\
&  =\left\|  \frac{I(t)}{t}\int_{t}^{1/2}\left|  f(s)\right|  \frac{ds}%
{I(s)}\right\|  _{\bar{X}}\text{ (since }Q_{I}\left|  f\right|  (t)\text{ is
decreasing)}\\
&  =\left\|  \frac{I(t)}{t}\int_{t}^{1/2}\left|  f(s)\right|  \frac{s}%
{I(s)}\frac{ds}{s}\right\|  _{\bar{X}}\text{ }\\
&  \leq\left\|  \int_{t}^{1/2}\left|  f(s)\right|  \frac{ds}{s}\right\|
_{\bar{X}}\text{ \ \ (since }\frac{s}{I(s)}\text{ decreases)}\\
&  \preceq\left\|  f\right\|  _{\bar{X}}\text{ \ (since }\underline{\alpha
}_{X}>0\text{).}%
\end{align*}

Case $2:$ The convex isoperimetric estimator satisfies (\ref{peso}).

Consider $\tilde{Q}_{I}$ defined by
\[
\tilde{Q}_{I}f(t)=\frac{I(t)}{t}Q_{I}f(t).
\]

We claim that $\tilde{Q}_{I}:L^{1}(0,1)\rightarrow L^{1}(0,1)$ is
bounded, and $\tilde{Q}_{I}:L^{\infty}(0,1)\rightarrow
L^{\infty}(0,1)$ is bounded, then by interpolation (see \cite{KPS})
$\tilde{Q}_{I}$ will be bounded on
$\bar{X}$. Thus%
\begin{align*}
\left\|  Q_{I}f\right\|  _{Z}  &  \leq\left\|  Q_{I}\left|  f\right|
\right\|  _{Z}=\left\|  \frac{I(t)}{t}\left(  Q_{I}\left|  f\right|  \right)
^{\ast}(t)\right\|  _{\bar{X}}\\
&  =\left\|  \frac{I(t)}{t}Q_{I}\left|  f\right|  (t)\right\|  _{\bar{X}%
}\text{ \ (since }Q_{I}\left|  f\right|  (t)\text{ is decreasing)}\\
&  =\left\|  \tilde{Q}_{I}\left|  f\right|  (t)\right\|  _{\bar{X}}\\
&  \preceq\left\|  f\right\|  _{\bar{X}}
\end{align*}
and Theorem \ref{t1} applies.

We going to prove now the claim.

 By the convexity of $I,$ $\frac{I(t)}{t}$ is
increasing for $0<t<1/2,$ thus
\[
\int_{0}^{s}\frac{I(t)}{t}dt\leq I(s),
\]
therefore
\begin{align*}
\left\|  \tilde{Q}_{I}f\right\|  _{1}  &  \leq\int_{0}^{1}\tilde{Q}_{I}\left(
\left|  f\right|  \right)  (t)dt\\
&  =\int_{0}^{1/2}\frac{I(t)}{t}\left(  \int_{t}^{1/2}\left|  f(s)\right|
\frac{ds}{I(s)}\right)  dt\\
&  =\int_{0}^{1/2}\frac{\left|  f(s)\right|  }{I(s)}\left(  \int_{0}^{s}%
\frac{I(t)}{t}dt\right)  ds\\
&  \leq\int_{0}^{1/2}\left|  f(s)\right|  ds\\
&  =\left\|  f\right\|  _{1}.
\end{align*}
Similarly$,$
\begin{align*}
\left\|  \tilde{Q}_{I}f\right\|  _{\infty}  &  \leq\sup_{0<t<1}\tilde{Q}%
_{I}\left(  \left|  f\right|  \right)(t) \\
&  \leq\sup_{0<t<1/2}\frac{I(t)}{t}\int_{t}^{1/2}\left|  f(s)\right|
\frac{ds}{I(s)}\\
&  \leq\left\|  f\right\|  _{\infty}\sup_{0<t<1/2}\left(  \frac{I(t)}{t}%
\int_{t}^{1/2}\frac{ds}{I(s)}\right) \\
&  \leq c\left\|  f\right\|  _{\infty}\text{ (by (\ref{peso})).}%
\end{align*}

To finish the proof of Theorem it remains to see that
\begin{equation}
\left\|  f_{\mu}^{\ast}\right\|  _{\bar{Y}}\preceq\left\|  f_{\mu}^{\ast
}(t)\frac{I(t)}{t}\right\|  _{\bar{X}}. \label{incl0}%
\end{equation}
Let $C_{\bar{Y}}$ be the constant quasi-norm of $\bar{Y}$, then
\begin{align}
\left\|  f_{\mu}^{\ast}\right\|  _{\bar{Y}}  &  =\left\|  f_{\mu}^{\ast
}(t)\chi_{(0,1/4)}(t)+f_{\mu}^{\ast}(t)\chi_{(1/4,1/2)}(t)+f_{\mu}^{\ast
}(t)\chi_{(1/2,3/4)}(t)+f_{\mu}^{\ast}(t)\chi_{(3/4,1)}(t)\right\|  _{\bar{Y}%
}\label{pri}\\
&  \leq4C_{\bar{Y}}^{2}\left\|  f_{\mu}^{\ast}(t)\chi_{(0,1/4)}(t)\right\|
_{\bar{Y}}.\nonumber
\end{align}
Since $f_{\mu}^{\ast}$ is decreasing,
\[
f_{\mu}^{\ast}(t)\chi_{(0,1/4)}(t)\leq\frac{1}{\ln2}\int_{t/2}^{t}f_{\mu
}^{\ast}(s)\frac{ds}{s}=\frac{1}{\ln2}\int_{t/2}^{1/2}f_{\mu}^{\ast}%
(s)\chi_{(0,1/4)}(s)\frac{I(s)}{s}\frac{ds}{I(s)}.
\]
Thus
\begin{align}
\left\|  f_{\mu}^{\ast}(t)\chi_{(0,1/4)}(t)\right\|  _{\bar{Y}}  &
\preceq\left\|  Q_{I}\left(  f_{\mu}^{\ast}(\cdot)\chi_{(0,1/4)}(\cdot
)\frac{I(\cdot)}{\cdot}\right)  (t/2)\right\|  _{\bar{Y}}\label{sec}\\
&  \preceq\left\|  f_{\mu}^{\ast}(t/2)\chi_{(0,1/4)}(t/2)\frac{I(t/2)}%
{t/2}\right\|  _{\bar{X}}\text{ (by (\ref{hip}))}\nonumber\\
&  \preceq\left\|
f_{\mu}^{\ast}(t)\chi_{(0,1/2)}(t)\frac{I(t)}{t}\right\|
_{\bar{X}}\text{ }\nonumber\\
&  \preceq\left\|  f_{\mu}^{\ast}(t)\frac{I(t)}{t}\right\|  _{\bar{X}%
}.\nonumber
\end{align}
Combining (\ref{pri}) and (\ref{sec}) we obtain (\ref{incl0}).
\end{proof}

\begin{remark}
\label{positiva}If $g\in Lip(\Omega)$ is positive with $m(g)=0,$ then it
follows from the previous Theorem that%
\[
\left\|  g_{\mu}^{\ast}(t)\frac{I(t)}{t}\right\|  _{\bar{X}}\preceq\left\|
\left|  \nabla g\right|  \right\|  _{X},
\]
\end{remark}

\subsection{Nash inequalities.}

In this section we obtain Nash type inequalities. We will focus in the
following type of probability measures.

\begin{definition}
Let $\mu$ be a probability measure on $\Omega,$ which admits a
convex isoperimetric estimator $I.$

\begin{enumerate}
\item  Let $\alpha>0.$ We will say that $\mu$ is $\alpha-$Cauchy type if
\[
I(t)=c_{\mu}\min(t,1-t)^{1+1/\alpha}.
\]

\item  Let $0<p<1.$ We will say that $\mu$ is a extended $p-$sub-exponential type if
\[
I(t)=c_{\mu}\min(t,1-t)\left(  \log\frac{1}{\min\left(  t,1-t\right)
}\right)  ^{1-1/p}.
\]
\end{enumerate}

In both cases $c_{\mu}$ denotes a positive constant.
\end{definition}

\begin{theorem}
\label{nash}The following Nash inequalities holds:

\begin{enumerate}
\item  Let $\mu$ be  $\alpha-$Cauchy type. Let $X$ be a r.i. space on
$\Omega$ with $\underline{\alpha}_{X}>0.$ Let $1<q\leq\infty$ such that
$0\leq1/q<$ $\underline{\alpha}_{X}.$ Then for all $f\in Lip(\Omega)$ positive
with $m(f)=0,$ we have%
\[
\left\|  f\right\|  _{X}\preceq\min_{r>1}\left(  r\left\|  \left|  \nabla
f\right|  \right\|  _{X}+\left\|  f\right\|  _{q,\infty}\phi_{X}(r^{-\alpha
})r^{\alpha/q}\right)  .
\]

\item  Let $\mu$ be extended $p-$sub-exponential type. Let $X$ be a r.i. space on
$\Omega$ $.$ Let $\beta>0.$ Then for all $f\in Lip(\Omega)$ positive with
$m(f)=0,$ we have%
\[
\left\|  f\right\|  _{X}\preceq\left\|  \left|  \nabla f\right|
\right\| _{X}^{\frac{\beta}{\beta+1}}\left\|  f\right\|
_{{X(\ln(\frac{1}{t})^{\beta(\frac{1}{p}-1)}})}^{\frac{1}{\beta+1}}.
\]
\end{enumerate}

\begin{proof}
Part $1$. Let $f\in Lip(\Omega)$ positive with $m(f)=0$ and let $\omega
(t)=t^{-1/\alpha}$ $\left(  0<t<1/2\right)  .$ Let $r>1$ and let $\beta>0$
that will be chosen later. Then%
\begin{align}
\left\|  f\right\|  _{X}  &  =\left\|  f_{\mu}^{\ast}\right\|  _{\bar{X}}%
\leq\left\|  f_{\mu}^{\ast}(t)\frac{\omega(t)}{\omega(t)}\chi_{\left\{
\omega<r)\right\}  }(t)\right\|  _{\bar{X}}+\left\|  f_{\mu}^{\ast}(t)\left(
\frac{\omega(t)}{\omega(t)}\right)  ^{\beta}\chi_{\left\{  \omega>r\right\}
}(t)\right\|  _{\bar{X}}\label{paso1}\\
&  \leq r\left\|  f_{\mu}^{\ast}(t)t^{1/\alpha}\right\|  _{\bar{X}}+r^{-\beta
}\left\|  f_{\mu}^{\ast}(t)t^{-\beta/\alpha}\chi_{\left(  0,r^{-\alpha
}\right)  }(t)\right\|  _{\bar{X}}\nonumber\\
&  =r\left\|  f_{\mu}^{\ast}(t)t^{1/\alpha}\right\|  _{\bar{X}}+r^{-\beta
}\left\|  t^{1/q}f_{\mu}^{\ast}(t)t^{-\beta/\alpha-1/q}\chi_{\left(
0,r^{-\alpha}\right)  }(t)\right\|  _{\bar{X}}\nonumber\\
&  \leq r\left\|  f_{\mu}^{\ast}(t)t^{1/\alpha}\right\|  _{\bar{X}}+r^{-\beta
}\sup_{t>0}(t^{1/q}f_{\mu}^{\ast}(t))\left\|  t^{-\beta/\alpha-1/q}%
\chi_{\left(  0,r^{-\alpha}\right)  }(t)\right\|  _{\bar{X}}\nonumber\\
&  \leq r\left\|  f_{\mu}^{\ast}(t)t^{1/\alpha}\right\|  _{\bar{X}}+r^{-\beta
}\left\|  f\right\|  _{q,\infty}\left\|  t^{-\beta/\alpha-1/q}\chi_{\left(
0,r^{-\alpha}\right)  }(t)\right\|  _{\Lambda(\bar{X})}\text{ (by
(\ref{tango}))}\nonumber\\
&  \preceq r\left\|  f_{\mu}^{\ast}(t)t^{1/\alpha}\right\|  _{\bar{X}%
}+r^{-\beta}\left\|  f\right\|  _{q,\infty}\int_{0}^{r^{-\alpha}}%
t^{-\beta/\alpha-1/q}\frac{\phi_{X}(t)}{t} \text{ (by Lemma \ref{indices})}\nonumber\\
&  =r\left\|  f_{\mu}^{\ast}(t)t^{1/\alpha}\right\|
_{\bar{X}}+r^{-\beta }\left\|  f\right\|  _{q,\infty}J(r).\nonumber
\end{align}
Let $0\leq1/q<\gamma<$ $\underline{\alpha}_{X},$ by Lemma \ref{indices},%
\[
\int_{0}^{r^{-\alpha}}t^{-\beta/\alpha-1/q}\frac{\phi_{X}(t)}{t^{\gamma
}t^{1-\gamma}}\preceq\frac{\phi_{X}(r^{-\alpha})}{r^{-\alpha\gamma}}\int
_{0}^{r^{-\alpha}}t^{-\beta/\alpha-1/q+\gamma-1}.
\]
At this stage we select $0<\beta<\alpha\left(  \gamma-1/q\right)  ,$ then
\[
\int_{0}^{r^{-\alpha}}t^{-\beta/\alpha-1/q+\gamma-1}\preceq r^{-\alpha\left(
-\beta/\alpha-1/q+\gamma\right)  },
\]
thus%
\[
J(r)\preceq\phi_{X}(r^{-\alpha})r^{\beta+\alpha/q}.
\]
Inserting this information in (\ref{paso1}) and by Remark \ref{positiva}, we
get%
\begin{align*}
\left\|  f\right\|  _{X}  &  \preceq r\left\|  f_{\mu}^{\ast}(t)t^{1/\alpha
}\right\|  _{\bar{X}}+\left\|  f\right\|  _{q,\infty}\phi_{X}(r^{-\alpha
})r^{\alpha/q}\\
&  \preceq r\left\|  \left|  \nabla f\right|  \right\|  _{X}+\left\|
f\right\|  _{q,\infty}\phi_{X}(r^{-\alpha})r^{\alpha/q}\text{.}%
\end{align*}
Part $2.$ Let $f\in Lip(\Omega)$ positive with $m(f)=0$ and let $\omega
(t)=\left(  \ln\frac{1}{t}\right)  ^{\frac{1}{p}-1}$ $\left(  0<t<1/2\right)
.$ Let $r>1$ and $\beta>0$.
\begin{align}
\left\|  f\right\|  _{X}  &  =\left\|  f_{\mu}^{\ast}\right\|  _{\bar{X}}%
\leq\left\|  f_{\mu}^{\ast}(t)\frac{\omega(t)}{\omega(t)}\chi_{\left\{
\omega<r)\right\}  }(t)\right\|  _{\bar{X}}+\left\|  f_{\mu}^{\ast}(t)\left(
\frac{\omega(t)}{\omega(t)}\right)  ^{\beta}\chi_{\left\{  \omega>r\right\}
}(t)\right\|  _{\bar{X}}\nonumber\\
&  \leq r\left\|  f_{\mu}^{\ast}(t)\left(  \ln\frac{1}{t}\right)  ^{1-\frac
{1}{p}}\right\|  _{\bar{X}}+r^{-\beta}\left\|  f_{\mu}^{\ast}(t)\left(
\ln\frac{1}{t}\right)  ^{\beta\left(  \frac{1}{p}-1\right)  }\right\|
_{\bar{X}}\nonumber\\
\nonumber\\
&  \preceq r\left\|  \left|  \nabla f\right|  \right\|
_{X(\log(\frac{1}{t})^{\beta(\frac{1}{p}-1)})}+r^{-\beta }\left\|
f\right\| _{X}\text{ \ (by Remark \ref{positiva}).}\nonumber
\end{align}
We finish taking the $\inf$ for $r>1$.
\end{proof}
\end{theorem}

\begin{remark}
Let $X$ be a r.i. space on $\Omega$ with $\underline{\alpha}_{X}>0.$ Let
$1<q\leq\infty$ such that $0\leq1/q<$ $\underline{\alpha}_{X}.$ Then%
\[
L^{q,\infty}\left(  \Omega\right)  \subset\Lambda(X)\subset X\left(
\Omega\right)  .
\]
Effectively, by Lemma \ref{indices}%
\[
\left\|  f\right\|  _{\Lambda(X)}=\int_{0}^{1}f^{\ast}(t)\frac{\phi_{X}(t)}%
{t}dt\leq\left\|  f\right\|  _{q,\infty}\int_{0}^{1}\frac{\phi_{X}%
(t)}{t^{1+1/q}}dt.
\]
The last integral is finite since taking $0\leq1/q<\gamma<$ $\underline
{\alpha}_{X},$ we get%
\[
\int_{0}^{1}\frac{\phi_{X}(t)}{t^{1+1/q}}dt=\int_{0}^{1}t^{1/q+\gamma-1}%
\frac{\phi_{X}(t)}{t^{\gamma}}dt\preceq\int_{0}^{1}t^{1/q+\gamma-1}<\infty.
\]
\end{remark}

\section{Examples and applications\label{five}}

In this section we will apply the previous work to the probability measures
introduced in examples \ref{ex1}, \ref{ex2} and \ref{ex3}.

\subsection{Cauchy type laws}

Consider the probability measure space $(\mathbb{R}^{n},d,\mu)$ where $d$ is
the Euclidean distance and $\mu$ is the probability measure introduced in
Example \ref{ex1}. Such measures have been introduced by Borell \cite{Bo} (see
also \cite{bob}). Prototypes of these probability measures are the generalized
Cauchy distributions\footnote{These measures are Barenblatt solutions of the
porous medium equations and appear naturally in weighted porous medium
equations, giving the decay rate of this nonlinear semigroup towards the
equilibrium measure, see \cite{Va} and \cite{DGGW}.}:%
\[
d\mu(x)=\frac{1}{Z}\left(  \left(  1+\left|  x\right|  ^{2}\right)
^{1/2}\right)  ^{-(n+\alpha)},\text{ \ }\alpha>0.
\]

A convex isoperimetric estimator for these measures is (see \cite[Proposition
4.3]{CGGR}):
\[
I(t)=\min(t,1-t)^{1+1/\alpha}.
\]
Obviously for $0<t<1/2,$ we have%
\[
\int_{t}^{1/2}\frac{ds}{s^{1+1/\alpha}}\preceq\frac{t}{t^{1+1/\alpha}}.
\]

Thus by Theorem \ref{t2}, given a r.i. space $X$ on $\mathbb{R}^{n}$
we get
\[
\inf_{c\in\mathbb{R}}\left\|  \left(  g-c\right)  _{\mu}^{\ast}\frac
{\min(t,1-t)^{1+1/\alpha}}{t}\right\|  _{\bar{X}}\preceq\left\|  \left|
\nabla g\right|  \right\|  _{X},\text{ \ \ \ }\left(  g\in Lip(\mathbb{R}%
^{n})\right)  .
\]

\begin{proposition}
Let $1\leq p<\infty,$ $1\leq q\leq\infty.$ \ For all $f\in Lip(\mathbb{R}%
^{n})$ positive with $m(f)=0,$ we get

\begin{enumerate}
\item
\[
\left\|  f\right\|  _{\frac{p\alpha}{p+\alpha},q}\preceq\left\|
\left| \nabla f\right|  \right\|  _{p,q}.
\]

\item  For all $s>p$
\[
\left\|  f\right\|  _{p,q}\preceq\left\|  \left|  \nabla f\right|
\right\|
_{p,q}^{\frac{\beta}{\beta+1}}\left\|  f\right\|  _{s,\infty}^{\frac{1}%
{\beta+1}}%
\]
where $\beta=\alpha(\frac{1}{p}-\frac{1}{s}).$
\end{enumerate}
\end{proposition}

\begin{proof}
1) By Theorem \ref{t2} we get
\[
\left\|  f_{\mu}^{\ast}t^{\frac{1}{\alpha}}\right\|
_{p,q}\preceq\left\| \left|  \nabla f\right|  \right\|  _{p,q}.
\]
Now by \cite[Page 76]{KPS} we have that
\[
\left\|  f_{\mu}^{\ast}t^{\frac{1}{\alpha}}\right\|  _{p,q}^{q}=\int_{0}%
^{1}\left[  \left(  t^{\frac{1}{\alpha}}f_{\mu}^{\ast}(t)\right)
^{\ast }t^{\frac{1}{p}}\right]
^{q}\frac{dt}{t}\simeq\int_{0}^{1}\left(  t^{\frac
{1}{\alpha}+\frac{1}{p}}f_{\mu}^{\ast}(t)\right)
^{q}\frac{dt}{t}=\left\| f\right\|
_{\frac{p\alpha}{p+\alpha},q}^{q}.
\]

2) is a direct application of Theorem \ref{nash}.
\end{proof}

\begin{remark}
If in the previous Proposition we take $p=q=1,$ we obtain%
\begin{equation}
\left\|  f\right\|  _{\frac{\alpha}{\alpha+1},1}\preceq\left\|
\left|  \nabla f\right|  \right\|  _{1}.\label{mil01}
\end{equation}
If $\frac{1}{q}=\frac{1}{p}+\frac{1}{\alpha},$ then we get%
\begin{equation}
\left\|  f\right\|  _{p,\frac{p(1+\alpha)}{\alpha}}\preceq\left\|
\left| \nabla f\right|  \right\|
_{q,\frac{p(1+\alpha)}{\alpha}}.\label{mil02}
\end{equation}
For $p\geq 1$ and $s=\infty$, we have that
\begin{equation}
\left\|  f\right\|  _{p}\preceq\left\|  \left|  \nabla f\right|
\right\|
_{p}^{\frac{\beta}{\beta+1}}\left\|  f\right\|  _{\infty}^{\frac{1}%
{\beta+1}}.\label{mil03}
\end{equation}
Inequalities \ref{mil01} and \ref{mil02} were proved in
\cite[Proposition 5.13]{Mil}. Inequality \ref{mil03} was obtained in
\cite[Proposition 5.15]{Mil}.
\end{remark}

We close this section with the following optimality result:

\begin{theorem}
\label{alph}
Let $\alpha>0.$ Let $\bar{X}$ be a r.i. space on $(0,1)$ and let $Z$
be a q.r.i. space on $(0,1).$ Assume that for any probability
measure $\mu$ be of $\alpha-$Cauchy type in $\mathbb{R}^{n},$
there is a $C_{\mu}>0,$ such that for all $f\in Lip(\mathbb{R}^{n})$ positive
with $m(f)=0,$ we get
\[
\left\|    f  _{\mu}^{\ast}\right\|
_{Z}\leq C_{\mu}\left\|  \left|  \nabla f\right|  _{\mu}^{\ast}\right\|
_{\bar{X}}.
\]
Then for all $g\in Lip(\mathbb{R}^{n})$%
\[
\left\|    g  _{\mu}^{\ast}\right\|
_{Z}\preceq\left\| g _{\mu}^{\ast
}(t)\frac{I(t)}{t}\right\|  _{\bar{X}}%
\]
\end{theorem}

\begin{proof}
Let $\mu$ be the Cauchy probability measure on $\mathbb{R}$ defined by%
\[
d\mu(s)=\frac{\alpha}{2\left(  1+\left|  s\right|  ^{2}\right)  ^{\frac
{1+\alpha}{2}}}ds=\varphi(s)dx,\text{ \ \ \ }s\in\mathbb{R}\text{.}%
\]
It is known (see \cite[Proposition 5.27]{CGGR} and \cite{FPR}) that
its isoperimetric profile is given by
\[
I_{\mu}(t)=\varphi\left(  H^{-1}(t\right)  )=\alpha2^{1/\alpha}\min
(t,1-t)^{1+1/\alpha},\text{ \ \ \ }t\in\lbrack0,1],
\]
where $H$ is the distribution function of $\mu$, i.e. $H:\mathbb{R}%
\rightarrow(0,1)$ is the increasing function given by
\[
H(r)=\int_{-\infty}^{r}\varphi(t)dt.
\]

Consider on $\mathbb{R}^{n}$ the product measure $\mu^{n}$, by Proposition
5.27 of \cite{CGGR} the function%
\[
I(t)=\frac{c_{\alpha}}{n^{1/\alpha}}\min(t,1-t)^{1+1/\alpha}%
\]
is a convex isoperimetric estimator of $\mu^{n}$ ($c_{\alpha}$ denotes a
positive constant depending only of $\alpha).$

Given a\ positive measurable function $f\ $with $suppf$ $\subset(0,1/2),$
consider
\[
F(t)=\int_{t}^{1}f(s)\frac{ds}{I_{\mu}(s)},\text{ \ \ \ }t\in(0,1),
\]
and define
\[
u(x)=F(H(x_{1})),\text{ \ \ \ \ \ }x\in\mathbb{R}^{n}.
\]
Then,%
\[
\left|  \nabla u(x)\right|  =\left|  \frac{\partial}{\partial x_{1}%
}u(x)\right|  =\left|  -f(H(x_{1}))\frac{H^{\prime}(x_{1})}{I_{\mu}(H(x_{1}%
))}\right|  =f(H(x_{1})).
\]
Let $A$ be a Young's function and let $s=H(x_{1})$. Then,
\begin{align*}
\int_{\mathbb{R}^{n}}A(f(H(x_{1})))d\mu^{n}(x)  &  =\int_{\mathbb{R}%
}A(f(H(x_{1})))d\mu(x_{1})\\
&  =\int_{0}^{1}A(f(s))ds.
\end{align*}
Therefore, by \cite[exercise 5 pag. 88]{BS}
\[
\left|  \nabla u\right|  _{\mu^{n}}^{\ast}(t)=f^{\ast}(t).
\]
Similarly%
\[
u_{\mu^{n}}^{\ast}(t)=\int_{t}^{1}f(s)\frac{ds}{I_{\mu}(s)}.
\]
Since $m(u)=0$, by the hypothesis we get
\begin{align*}
\left\|  \int_{t}^{1}f(s)\frac{ds}{I_{\mu}(s)}\right\|  _{Z}  &  =\left\|
u_{\mu^{n}}^{\ast}\right\|  _{Z}\\
&  \leq C_{\mu^{n}}\left\|  \left|  \nabla f\right|  _{\mu^{n}}^{\ast
}\right\|  _{\bar{X}}\\
&  =C_{\mu^{n}}\left\|  f^{\ast}(t)\right\|  _{\bar{X}}\\
&  =C_{\mu^{n}}\left\|  f\right\|  _{\bar{X}}.
\end{align*}
Finally, from
\[
I_{\mu}(t)=\frac{\alpha2^{1/_{\alpha}}n^{1/\alpha}}{c_{\alpha}}I(t)
\]
we have that
\[
\left\|  Q_{I}f\right\|  _{Z}\leq\frac{c_{\alpha}C_{\mu^{n}}}{\alpha
2^{1/_{\alpha}}n^{1/\alpha}}\left\|  f\right\|  _{\bar{X}}%
\]
and the results follows by Theorem \ref{t2}.
\end{proof}

\subsection{Extended $p-$sub-exponential law}
Consider the probability measure on $\mathbb{R}^{n}$ defined by
\[
d\mu(x)=\frac{1}{Z_{p}}e^{-V(x)^{p}}dx=\varphi(x)dx
\]
for some positive convex function $V:\mathbb{R}^{n}\to (0,\infty)$
and $p\in(0,1)$.

A typical example is $V(x) = |x|^p$, and $0 < p < 1$, which yields
to sub-exponential type law.

A convex isoperimetric estimator for these type of measures is (see
\cite[Proposition 4.5]{CGGR} and \cite{FPR}):
\[
I(t)=c_{p}\min\left(  t,1-t\right)  \left(  \log\frac{1}{\min\left(
t,1-t\right)  }\right)  ^{1-1/p}.
\]

By Theorem \ref{t2}, given a r.i. space $X$ on $\mathbb{R}^{n}$ with
$\underline{\alpha}_{X}>0,$ we get
\[
\inf_{c\in\mathbb{R}}\left\|  \left(  g-c\right)  _{\mu}^{\ast}\frac{c_{p}%
\min\left(  t,1-t\right)  \left(  \log\frac{1}{\min\left(
t,1-t\right) }\right)  ^{1-1/p}}{t}\right\|
_{\bar{X}}\preceq\left\|  \left|  \nabla
g\right|  \right\|  _{X},\text{ \ \ \ }\left(  g\in Lip(\mathbb{R}%
^{n})\right)  .
\]

In the particular case that $X=L^{r,q}$ we obtain

\begin{proposition}
Let $1\leq r<\infty,$ $1\leq q<\infty.$ \ For all $f\in
Lip(\mathbb{R}^{n})$ positive with $m(f)=0,$ we get

\begin{enumerate}
\item
\[
\left\|  f\right\|  _{L^{r,q}(\log L)^{1-1/p}}\preceq\left\|  \left|
\nabla f\right|  \right\|  _{r,q}.
\]

\item  For all $\beta>0$
\[
\left\|  f\right\|  _{r,q}\preceq\left\|  \left|  \nabla f\right|
\right\| _{r,q}^{\frac{\beta}{\beta+1}}\left\|  f\right\|
_{L^{r,q}(\log
L)^{\beta(1-1/p)}}^{\frac{1}{\beta+1}}%
\]
\end{enumerate}
\end{proposition}

\begin{theorem}
Let $p\in(0,1).$ Let $\bar{X}$ be a r.i. space on $(0,1)$ and let
$Z$ be a q.r.i. space on $(0,1).$ Assume that for any extended $p-$
sub-exponential law $\mu$ in $\mathbb{R}^{n},$ there is a
$C_{\mu}>0,$ such that for all $f\in Lip(\mathbb{R}^{n})$ positive
with $m(f)=0,$ we get
\[
\left\|  f_{\mu}^{\ast}\right\|  _{Z}\leq C_{\mu}\left\|  \left|
\nabla f\right|  _{\mu}^{\ast}\right\|  _{\bar{X}}.
\]
Then, for all $g\in Lip(\mathbb{R}^{n})$, we get
\[
\left\|  g_{\mu}^{\ast}\right\|  _{Z}\preceq\left\|  g_{\mu}^{\ast}%
(t)\frac{I(t)}{t}\right\|  _{\bar{X}}%
\]
\end{theorem}

\begin{proof}
Let $\mu$ be the probability measure on $\mathbb{R}$ with density%
\[
d\mu(s)=\frac{e^{-\left|  s\right|
^{p}}}{Z_{p}}ds=\varphi(s)ds,\text{
\ \ \ }s\in\mathbb{R}\text{.}%
\]
Its isoperimetric profile is (see \cite[Proposition 5.25]{CGGR})
\[
I_{\mu}(t)=\varphi\left(  H^{-1}(t\right)  )=c_{p}\min\left(
t,1-t\right) \left(  \log\frac{1}{\min\left(  t,1-t\right)  }\right)
^{1-1/p},\text{ \ \ \ }t\in\lbrack0,1],
\]
where $H$ is the distribution function of $\mu$, i.e. $H:\mathbb{R}%
\rightarrow(0,1)$ is  defined by
\[
H(r)=\int_{-\infty}^{r}\varphi(t)dt.
\]
Consider on $\mathbb{R}^{n}$ the product measure $\mu^{n}$, by
Proposition 5.25 of \cite{CGGR}, there exists a positive constant
$c$ such that the
function%
\[
I(t)=c\min\left(  t,1-t\right)  \left(  \log\frac{n}{\min\left(
t,1-t\right) }\right)  ^{1-1/p},
\]
is a convex isoperimetric estimator of $\mu^{n}$

Let $f$ be a\ positive measurable function $f\ $with $suppf$
$\subset(0,1/2),$ consider
\[
F(t)=\int_{t}^{1}f(s)\frac{ds}{I_{\mu}(s)},\text{ \ \ \ }t\in(0,1),
\]
and define
\[
u(x)=F(H(x_{1})),\text{ \ \ \ \ \ }x\in\mathbb{R}^{n}.
\]
Using the same method that in Theorem \ref{alph}, we obtain
\[
\left|  \nabla u\right|  _{\mu^{n}}^{\ast}(t)=f^{\ast}(t) \text{ and
} u_{\mu^{n}}^{\ast}(t)=\int_{t}^{1}f(s)\frac{ds}{I_{\mu}(s)}.
\]
Since $m(u)=0$, by the hypothesis we get
\begin{align*}
\left\|  \int_{t}^{1}f(s)\frac{ds}{I_{\mu}(s)}\right\|  _{Z} &
=\left\|
u_{\mu^{n}}^{\ast}\right\|  _{Z}\\
&  \leq C_{\mu^{n}}\left\|  \left|  \nabla f\right|
_{\mu^{n}}^{\ast
}\right\|  _{\bar{X}}\\
&  =C_{\mu^{n}}\left\|  f^{\ast}(t)\right\|  _{\bar{X}}\\
&  =C_{\mu^{n}}\left\|  f\right\|  _{\bar{X}}.
\end{align*}
Finally, from
\[
I_{\mu}(t)\simeq I(t)
\]
we have that
\[
\left\|  Q_{I}f\right\|  _{Z}\preceq\left\|  f\right\|  _{\bar{X}}.
\]
and Theorem \ref{t2} applies.
\end{proof}

\subsection{Weighted Riemannian manifold with negative dimension.}
Let $\left(  M^{n},g,\mu\right)  $ be a $n-$dimensional weighted
Riemannian manifold ($n\geq2)$ that satisfies the $CD(0,N)$
curvature condition with $N<0.$ (See \cite[Secction 5.4]{Mil}).

A convex isoperimetric estimator  is given by
\[
I(t)=\min(t,1-t)^{-1/N}.
\]

Obviously for $0<t<1/2,$ we have%
\[
\int_{t}^{1/2}\frac{ds}{s^{-1/N}}\preceq\frac{t}{t^{-1/N}}.
\]

Thus by Theorem \ref{t2}, given a r.i. space $X$ on $\mathbb{R}^{n}$
we get
\[
\inf_{c\in\mathbb{R}}\left\|  \left(  g-c\right)  _{\mu}^{\ast}\frac
{\min(t,1-t)^{-1/N}}{t}\right\|  _{\bar{X}}\preceq\left\| \left|
\nabla g\right|  \right\|  _{X},\text{ \ \ \ }\left(  g\in Lip(\mathbb{R}%
^{n})\right)  .
\]

In particular if $1\leq p<\infty,$ $1\leq q\leq\infty)$ and
$X=L^{p,q}$, then for all $f\in Lip(\mathbb{R}^{n})$ positive with
$m(f)=0,$ ($1\leq p<\infty,$ $1\leq q\leq\infty)$

\[
\left\|  f\right\|  _{\gamma,q}\preceq\left\|  \left|  \nabla f\right|  \right\|  _{p,q}.
\]
where $\gamma=\frac{Np}{N-p(N+1)}$  for any $p,q$ satisfying
$\frac{N}{N-1}\leq p\leq -N$ and
$\frac{1}{q}=\frac{1}{p}-\frac{1}{N}-1.$

And by Theorem \ref{nash}, we have that
\[
\left\|  f\right\|  _{p,q}\preceq\left\|  \left|  \nabla f\right|  \right\|
_{p,q}^{\frac{\beta}{\beta+1}}\left\|  f\right\|  _{s,\infty}^{\frac{1}%
{\beta+1}}%
\]
where $s>p$ and $\beta=\alpha(\frac{1}{p}-\frac{1}{s}).$

\end{document}